\documentclass[a4paper,11pt]{amsart}
\usepackage{mathptmx}

\title[Supplement to classification of three-fold divisorial contractions]{Supplement to classification of\\ three-fold divisorial contractions}

\author{Masayuki Kawakita}
\address{Research Institute for Mathematical Sciences, Kyoto University, Kyoto 606-8502, Japan}
\email{masayuki@kurims.kyoto-u.ac.jp}

\theoremstyle{plain}
\newtheorem{theorem}{Theorem}
\newtheorem{lemma}[theorem]{Lemma}
\newtheorem{corollary}[theorem]{Corollary}

\theoremstyle{remark}
\newtheorem{remark}[theorem]{Remark}
\newtheorem*{acknowledgements}{Acknowledgements}

\newcommand{\bC}{\mathbb{C}}
\newcommand{\bP}{\mathbb{P}}
\newcommand{\bZ}{\mathbb{Z}}
\newcommand{\cO}{\mathcal{O}}

\DeclareMathOperator{\ord}{ord}
\DeclareMathOperator{\wt}{wt}

\begin{document}
\begin{abstract}
Every three-fold divisorial contraction to a non-Gorenstein point is a weighted blow-up.
\end{abstract}

\maketitle

This supplement finishes the explicit description of a three-fold divisorial contraction whose exceptional divisor contracts to a non-Gorenstein point. Contractions to a quotient singularity were treated by Kawamata in \cite{Km96}. The author's study \cite{K05}, based on the singular Riemann--Roch formula, provided the classification except for the case of small discrepancy. On the other hand, Hayakawa classified those with discrepancy at most one in \cite{H99}, \cite{H00}, \cite{H05}, by the fact that there exist only a finite number of divisors with such discrepancy over a fixed singularity. The only case left was when it is a contraction to a c$D/2$ point with discrepancy two. We demonstrate its classification in Theorem \ref{thm:cD/2} by the method in \cite{K05}. It turns out that every contraction is a weighted blow-up.

\begin{theorem}\label{thm:nonGor}
Let $f\colon Y\to X$ be a three-fold divisorial contraction whose exceptional divisor $E$ contracts to a non-Gorenstein point $P$. Then $f$ is a weighted blow-up of the singularity $P\in X$ embedded into a cyclic quotient of a smooth five-fold.
\end{theorem}

Our method of the classification is to study the structure of the bi-graded ring $\bigoplus_{i,j}f_*\cO_Y(iK_Y+jE)/f_*\cO_Y(iK_Y+jE-E)$. We find local coordinates at $P$ to meet this structure and verify that $f$ should be a certain weighted blow-up. The choice of local coordinates is restricted by the action of the cyclic group, which makes easier the classification in the non-Gorenstein case. We do not know if this method is sufficient to settle all the remaining Gorenstein cases in \cite{K01}, \cite{K02}, \cite{K03} with discrepancy at most four.

\medskip
By a three-fold divisorial contraction to a point, we mean a projective morphism $f\colon(Y\supset E)\to(X\ni P)$ between terminal three-folds such that $-K_Y$ is $f$-ample and the exceptional locus $E$ is a prime divisor contracting to a point $P$. We shall treat $f$ on the germ at $P$ in the complex analytic category. According to \cite[Theorems 1.2, 1.3]{K05}, the only case left is
\begin{center}
type e1 with $P=\textrm{c}D/2$, the discrepancy $a/n=4/2$
\end{center}
in \cite[Table 3]{K05}. We shall prove the following theorem.

\begin{theorem}\label{thm:cD/2}
Suppose that $f$ is a divisorial contraction of type \textup{e1} to a c$D/2$ point with discrepancy $2$. Then $f$ is the weighted blow-up with $\wt(x_1,x_2,x_3,x_4,x_5)=(\frac{r+1}{2},\frac{r-1}{2},2,1,r)$ with $r\ge7$, $r\equiv\pm1$ mod $8$ for a suitable identification
\begin{align*}
P\in X\ \simeq\ o\in\bigg(
\begin{array}{c}
x_1^2+x_4x_5+p(x_2,x_3,x_4)=0\\
x_2^2+q(x_1,x_3,x_4)+x_5=0
\end{array}
\bigg)
\subset\bC^5_{x_1x_2x_3x_4x_5}/\frac{1}{2}(1,1,1,0,0),
\end{align*}
such that $p$ is of weighted order $>r$ and $q$ is weighted homogeneous of weight $r-1$ for the weights distributed above.
\end{theorem}

The proof is along the argument in \cite[Section 7]{K05}. Henceforth $f\colon(Y\supset E)\to(X\ni P)$ is a divisorial contraction of type e1 to a c$D/2$ point with discrepancy $2$. By \cite[Table 3]{K05}, $Y$ has only one singular point $Q$ say at which $E$ is not Cartier. $Q$ is a quotient singularity of type $\frac{1}{2r}(1,-1,r+4)$ with $r\ge7$, $r\equiv\pm1$ mod $8$.

We set vector spaces $V_i=V_i^0\oplus V_i^1$ with
\begin{align*}
V_i^0&:=f_*\cO_Y(-iE)/f_*\cO_Y(-(i+1)E),\\
V_i^1&:=f_*\cO_Y(K_Y-(i+2)E)/f_*\cO_Y(K_Y-(i+3)E).
\end{align*}
They are zero for negative $i$, and we have the (bi-)graded ring $\bigoplus V_i$ by a local isomorphism $\cO_X(2K_X)\simeq\cO_X$. To study its structure in lower-degree part, we first compute the dimensions of $V_i^j$ in terms of the finite sets
\begin{align*}
N_i:=\Bigl\{(l_1,l_2,l_3,l_4,l_5)\in\bZ_{\ge0}^5\Bigm|\frac{r+1}{2}l_1+\frac{r-1}{2}l_2+2l_3+l_4+rl_5=i,\ \ l_1,l_2\le1\Bigr\}.
\end{align*}
$N_i$ is decomposed into $N_i^0\sqcup N_i^1$ with $N_i^j:=\{(l_1,l_2,l_3,l_4,l_5)\in N_i\mid l_1+l_2+l_3\equiv j\ \textrm{mod}\ 2\}$.

\begin{lemma}\label{lem:dim}
$\dim V_i^j=\#N_i^j$.
\end{lemma}

\begin{proof}
We follow the notation in \cite{K05}. $(r_Q,b_Q,v_Q)=(2r,r+4,2)$ and $E^3=1/r$ by \cite[Tables 2, 3]{K05}. By $\dim V_i^j=d(j,-i-2j)$ for $i\ge-2$ in \cite[(2.8)]{K05}, the equality \cite[(2.6)]{K05} for $(j,-i-2j)$ implies that for $i\ge0$,
\begin{align*}
\dim V_i^j-\dim V_{i-2}^{1-j}=\frac{2i+1}{r}+B_{2r}(2i+rj+2)-B_{2r}(2i+rj).
\end{align*}
On the other hand, by $N_i^j=\bigl(N_{i-2}^{1-j}+(0,0,1,0,0)\bigr)\sqcup\{(l_1,l_2,0,l_4,l_5)\in N_i^j\}$,
\begin{align*}
\#N_i^j-\#N_{i-2}^{1-j}&=
\begin{cases}
\#\{(0,0,0,l_4,l_5)\in N_i^0\}+\#\{(1,1,0,l_4,l_5)\in N_i^0\}&\textrm{for $j=0$},\\
\#\{(0,1,0,l_4,l_5)\in N_i^1\}+\#\{(1,0,0,l_4,l_5)\in N_i^1\}&\textrm{for $j=1$}.
\end{cases}
\end{align*}
The lemma follows by verifying the coincidence of their right-hand sides.
\end{proof}

We shall find bases of $V_i$ starting with an arbitrary identification
\begin{align}\label{eqn:cD/2}
P\in X\ \simeq\ o\in(\phi=0)\subset\bC^4_{x_1x_2x_3x_4}/\frac{1}{2}(1,1,1,0).
\end{align}
For a semi-invariant function $h$, $\ord_Eh$ denotes the order of $h$ along $E$.

\begin{lemma}\label{lem:basis}
\begin{enumerate}
\item\label{itm:basis:1}
$\ord_E x_4=1$ and $\ord_Ex_i\ge2$ for $i=1,2,3$. There exists some $k$ with $\ord_Ex_k=2$. We set $x_k=x_3$ by permutation.
\item\label{itm:basis:2}
For $i<\frac{r-1}{2}$, the monomials $x_3^{l_3}x_4^{l_4}$ for $(0,0,l_3,l_4,0)\in N_i$ form a basis of $V_i$. In particular for $k=1,2$, $\ord_E\bar{x}_k\ge\frac{r-1}{2}$ for $\bar{x}_k:=x_k+\sum c_{kl_3l_4}x_3^{l_3}x_4^{l_4}$ with some $c_{kl_3l_4}\in \bC$, with summation over $(0,0,l_3,l_4,0)\in\bigcup_{i<\frac{r-1}{2}}N_i^1$.
\item\label{itm:basis:3}
There exists some $k$ with $\ord_E \bar{x}_k=\frac{r-1}{2}$ such that the monomials $\bar{x}_k$ and $x_3^{l_3}x_4^{l_4}$ for $(0,0,l_3,l_4,0)\in N_{\frac{r-1}{2}}$ form a basis of $V_{\frac{r-1}{2}}$. We set $\bar{x}_k=\bar{x}_2$ by permutation, then $\ord_E\hat{x}_1\ge\frac{r+1}{2}$ for $\hat{x}_1:=\bar{x}_1+\sum c_{l_2l_3l_4}\bar{x}_2^{l_2}x_3^{l_3}x_4^{l_4}$ with some $c_{l_2l_3l_4}\in \bC$, with summation over $(0,l_2,l_3,l_4,0)\in N_{\frac{r-1}{2}}^1$.
\item\label{itm:basis:4}
$\ord_E\hat{x}_1=\frac{r+1}{2}$. For $i<r-1$, the monomials $\hat{x}_1^{l_1}\bar{x}_2^{l_2}x_3^{l_3}x_4^{l_4}$ for $(l_1,l_2,l_3,l_4,0)\in N_i$ form a basis of $V_i$.
\item\label{itm:basis:5}
The monomials $\hat{x}_1^{l_1}\bar{x}_2^{l_2}x_3^{l_3}x_4^{l_4}$ for $(l_1,l_2,l_3,l_4,0)\in N_{r-1}^0$ have one non-trivial relation, say $\psi$, in $V_{r-1}^0$. The natural exact sequence below is exact.
\begin{align*}
0\to\bC\psi\to\bigoplus_{(l_1,l_2,l_3,l_4,0)\in N_{r-1}}\bC\hat{x}_1^{l_1}\bar{x}_2^{l_2}x_3^{l_3}x_4^{l_4}\to V_{r-1}\to0.
\end{align*}
\item\label{itm:basis:6}
$\ord_E\psi=r$. The natural exact sequence below is exact.
\begin{align*}
0\to\bC x_4\psi\to\bigoplus_{(l_1,l_2,l_3,l_4,l_5)\in N_r}\bC\hat{x}_1^{l_1}\bar{x}_2^{l_2}x_3^{l_3}x_4^{l_4}\psi^{l_5}\to V_r\to0.
\end{align*}
\end{enumerate}
\end{lemma}

\begin{proof}
We follow the proof of \cite[Lemma 7.2]{K05}, with using the computation of Lemma \ref{lem:dim}. (\ref{itm:basis:1}) follows from $\dim V_1^0=1$, $\dim V_1^1=0$ and $\dim V_2^1=1$. Then $V_4^0$ is spanned by $x_3^2$ and $x_4^4$, which should form a basis of $V_4^0$ by $\dim V_4^0=2$. Now (\ref{itm:basis:2}) to (\ref{itm:basis:5}) follow from the same argument as in \cite[Lemma 7.2]{K05}. We obtain the sequence in (\ref{itm:basis:6}) also, which is exact possibly except for the middle. Its exactness is verified by comparing dimensions.
\end{proof}

\begin{corollary}\label{cor:basis}
We distribute weights $\wt(\hat{x}_1,\bar{x}_2,x_3,x_4)=(\frac{r+1}{2},\frac{r-1}{2},2,1)$ to the coordinates $\hat{x}_1,\bar{x}_2,x_3,x_4$ obtained in Lemma \textup{\ref{lem:basis}}. Then $\phi$ in \textup{(\ref{eqn:cD/2})} is of form
\begin{align*}
\phi=cx_4\psi+\phi_{>r}(\hat{x}_1,\bar{x}_2,x_3,x_4)
\end{align*}
with $c\in\bC$ and a function $\phi_{>r}$ of weighted order $>r$, where $\psi$ is as in Lemma \textup{\ref{lem:basis}(\ref{itm:basis:5})}.
\end{corollary}

\begin{proof}
Decompose $\phi=\phi_{\le r}+\phi_{>r}$ into the part $\phi_{\le r}$ of weighted order $\le r$ and $\phi_{>r}$ of weighted order $>r$. Then $\ord_E\phi_{\le r}=\ord_E\phi_{>r}>r$, so $\phi_{\le r}$ is mapped to zero by the natural homomorphism
\begin{align*}
\bigoplus_{(l_1,l_2,l_3,l_4,0)\in \bigcup_{i\le r}N_i^0}\bC\hat{x}_1^{l_1}\bar{x}_2^{l_2}x_3^{l_3}x_4^{l_4}\to\cO_X/f_*\cO_Y(-(r+1)E),
\end{align*}
whose kernel is $\bC x_4\psi$ by Lemma \ref{lem:basis}(\ref{itm:basis:4})-(\ref{itm:basis:6}).
\end{proof}

We shall derive an expression of the germ $P\in X$ in Theorem \ref{thm:cD/2}. By \cite[Remark 23.1]{M85}, the c$D/2$ point $P\in X$ has an identification in (\ref{eqn:cD/2}) with $\phi$ either of
\begin{align}
\tag{A}\label{eqn:A}
\phi&=x_1^2+x_2x_3x_4+x_2^{2\alpha}+x_3^{2\beta}+x_4^\gamma,\\
\tag{B}\label{eqn:B}
\phi&=x_1^2+x_2^2x_4+\lambda x_2x_3^{2\alpha-1}+g(x_3^2,x_4),
\end{align}
with $\alpha,\beta\ge2$, $\gamma\ge3$, $\lambda\in\bC$ and $g\in(x_3^4,x_3^2x_4^2,x_4^3)$. As its general elephant has type $D_k$ with $k\ge2r$ by \cite[Lemma 5.2(i)]{K05}, we have
\begin{align}\label{eqn:elephant}
\gamma\ge r\ \ \textrm{in (\ref{eqn:A})},\qquad \ord g(0,x_4)\ge r\ \ \textrm{in (\ref{eqn:B})}.
\end{align}

\begin{lemma}\label{lem:A}
The case \textup{(\ref{eqn:A})} does not happen.
\end{lemma}

\begin{proof}
By Lemma \ref{lem:basis}(\ref{itm:basis:1}), $\ord_Ex_4=1$, $\ord_Ex_i\ge2$ for $i=1,2,3$ and some $\ord_Ex_i=2$. $\ord_Ex_1\ge3$ by the relation $-x_1^2=x_2x_3x_4+x_2^{2\alpha}+x_3^{2\beta}+x_4^\gamma$ and (\ref{eqn:elephant}). Thus we may set $\ord_Ex_3=2$ by permutation, and construct $\bar{x}_1,\bar{x}_2$ as in Lemma \ref{lem:basis}(\ref{itm:basis:2}).

Let $W_{\frac{r-1}{2}}$ be the subspace of $V_{\frac{r-1}{2}}$ spanned by the monomials in $x_3,x_4$. If $\bar{x}_1\not\in W_{\frac{r-1}{2}}$, the triple $(\bar{x}_1,x_3,x_4)$ plays the role of $(\bar{x}_2,x_3,x_4)$ in Lemma \ref{lem:basis}(\ref{itm:basis:3}). We construct $\hat{x}_2$ as in Lemma \ref{lem:basis}(\ref{itm:basis:3}) to obtain a quartuple $(\hat{x}_2,\bar{x}_1,x_3,x_4)$, and distribute $\wt(\hat{x}_2,\bar{x}_1,x_3,x_4)=(\frac{r+1}{2},\frac{r-1}{2},2,1)$ as in Corollary \ref{cor:basis}. Set $\bar{x}_1=x_1+p_1(x_3,x_4)$, $\hat{x}_2=x_2+p_2(\bar{x}_1,x_3,x_4)$ and rewrite $\phi$ as
\begin{align*}
\phi=(\bar{x}_1-p_1)^2+(\hat{x}_2-p_2)x_3x_4+(\hat{x}_2-p_2)^{2\alpha}+x_3^{2\beta}+x_4^\gamma.
\end{align*}
$\phi$ has the term $\bar{x}_1^2$ of weight $r-1$, which contradicts Corollary \ref{cor:basis}.

Hence $\bar{x}_1\in W_{\frac{r-1}{2}}$, and we obtain a quartuple $(\hat{x}_1,\bar{x}_2,x_3,x_4)$ by $\hat{x}_1=x_1+p_1(x_3,x_4)$, $\bar{x}_2=x_2+p_2(x_3,x_4)$ as in Lemma \ref{lem:basis}. Distribute $\wt(\hat{x}_1,\bar{x}_2,x_3,x_4)=(\frac{r+1}{2},\frac{r-1}{2},2,1)$ and rewrite $\phi$ as
\begin{align*}
\phi=(\hat{x}_1-p_1)^2+(\bar{x}_2-p_2)x_3x_4+(\bar{x}_2-p_2)^{2\alpha}+x_3^{2\beta}+x_4^\gamma.
\end{align*}
$\phi$ has the term $\bar{x}_2x_3x_4$ of weight $\frac{r+5}{2}$, whence $\frac{r+5}{2}\ge r$ by Corollary \ref{cor:basis}, a contradiction to $r\ge7$.
\end{proof}

\begin{lemma}\label{lem:B}
The germ $P\in X$ has an expression in Theorem \textup{\ref{thm:cD/2}}, with $q$ not of form $(x_3s(x_3^2,x_4))^2$, such that each $\ord_Ex_i$ coincides with $\wt x_i$ distributed in Theorem \textup{\ref{thm:cD/2}}.
\end{lemma}

\begin{proof}
We have the case (\ref{eqn:B}) by Lemma \ref{lem:A}. $\ord_Ex_4=1$ and $\ord_Ex_1\ge3$ as in (\ref{eqn:A}), then $\ord_Ex_2\ge3$ and $\ord_Ex_3=2$. We construct $\bar{x}_1,\bar{x}_2$ as in Lemma \ref{lem:basis}(\ref{itm:basis:2}). By the same reason as in the proof of Lemma \ref{lem:A}, we obtain $\bar{x}_1\in W_{\frac{r-1}{2}}$ and a quartuple $(\hat{x}_1,\bar{x}_2,x_3,x_4)$ by $\hat{x}_1=x_1+p_1(x_3,x_4)$, $\bar{x}_2=x_2+p_2(x_3,x_4)$. Distribute $\wt(\hat{x}_1,\bar{x}_2,x_3,x_4)=(\frac{r+1}{2},\frac{r-1}{2},2,1)$ and rewrite $\phi$ as
\begin{align*}
\phi=(\hat{x}_1-p_1)^2+(\bar{x}_2-p_2)^2x_4+\lambda(\bar{x}_2-p_2)x_3^{2\alpha-1}+g(x_3^2,x_4).
\end{align*}
$\phi$ has the term $\bar{x}_2^2x_4$ of weight $r$ and should be of form
\begin{align*}
\phi=(\bar{x}_2^2+h(\hat{x}_1,\bar{x}_2,x_3,x_4))x_4+\phi_{>r}(\hat{x}_1,\bar{x}_2,x_3,x_4)
\end{align*}
as in Corollary \ref{cor:basis} with $\psi=\bar{x}_2^2+h(\hat{x}_1,\bar{x}_2,x_3,x_4)$. In particular $p_2=0$ as otherwise $p_2\bar{x}_2x_4$ would be of weighted order $<r$, and one can write
\begin{align*}
\phi&=\hat{x}_1^2+x_4\psi+p(\bar{x}_2,x_3,x_4),\qquad \psi=\bar{x}_2^2+q(\hat{x}_1,x_3,x_4),
\end{align*}
where $p$ is of weighted order $>r$ and $q$ is weighted homogeneous of weight $r-1$. A desired expression is derived by setting $x_5:=-\psi$ and replacing $x_4$ with $-x_4$. $q$ is not of form $(x_3s(x_3^2,x_4))^2$ by Lemma \ref{lem:basis}(\ref{itm:basis:3}) and $\ord_E(\bar{x}_2^2+q)=r$.
\end{proof}

Take an expression of the germ $P\in X$ in Theorem \ref{thm:cD/2} by Lemma \ref{lem:B}. We shall apply the extension of \cite[Lemma 6.1]{K05} to the case when $X$ is embedded into a cyclic quotient of $\bC^5$. Let $g\colon Z\to X$ be the weighted blow-up with $\wt x_i=\ord_Ex_i$. By direct calculation, we verify the assumptions of \cite[Lemma 6.1]{K05} and that $Z$ is smooth outside the strict transform of $x_1x_2x_3x_4x_5=0$. Therefore $f$ should coincide with $g$ by \cite[Lemma 6.1]{K05}, and Theorem \ref{thm:cD/2} is completed. 

\begin{remark}
Using $H\cap E\simeq\bP^1$ in the proof of \cite[Theorem 5.4]{K05}, one can show that
\begin{enumerate}
\item
if $r\equiv1$ mod $8$, $x_2x_3^{(r+3)/4}$ appears in $p$ and $x_3^{(r-1)/2}$ appears in $q$,
\item
if $r\equiv7$ mod $8$, $x_3^{(r+1)/2}$ appears in $p$ and $x_1x_3^{(r-3)/4}$ appears in $q$.
\end{enumerate}
\end{remark}

Theorem \ref{thm:nonGor} follows from \cite{H99}, \cite{H00}, \cite{H05}, \cite{K05}, \cite{Km96} and Theorem \ref{thm:cD/2}.

{\small
\begin{acknowledgements}
I was motivated to write this supplement by a question of Professor J. A. Chen. He, with Professor T. Hayakawa, informed me that only one case was left. Partial support was provided by Grant-in-Aid for Young Scientists (A) 20684002.
\end{acknowledgements}
}

\bibliographystyle{amsplain}

\end{document}